\documentclass[letterpaper, 10 pt, conference]{ieeeconf}  

\IEEEoverridecommandlockouts                              

\usepackage{ntheorem}
\usepackage{hyperref}
\usepackage{textcomp}
\let\labelindent\relax
\usepackage{enumitem}
\usepackage{graphicx}
\usepackage{epstopdf}
\usepackage{float}
\usepackage[nocomma]{optidef}
\usepackage{dsfont}
\usepackage{times}
\usepackage{bm}
\usepackage{fancyhdr,graphicx,amsmath,amssymb}
\usepackage{algorithm} 
\usepackage{algpseudocode} 
\include{pythonlisting}
\usepackage{cite}
\usepackage{xcolor}
\usepackage{amssymb}
\usepackage{mathtools}

\makeatletter
\def\BState{\State\hskip-\ALG@thistlm}
\makeatother
\usepackage{stfloats}

\title{\LARGE \bf 
Robust Asynchronous and Network-Independent Cooperative Learning
}

\author{{Eduardo Mojica-Nava}\and
{David Yanguas-Rojas} \and 
{C\'esar A. Uribe}
\thanks{E. Mojica-Nava and D. Yanguas-Rojas are with the Department of Electrical and Electronics Engineering, Universidad Nacional de Colombia, Colombia. C. A. Uribe is with the Laboratory for Information and Decision Systems -LIDS, Massachusetts Institute of Technology -M.I.T., Cambridge, MA, USA. E-mails: \url{eamojican@unal.edu.co}, \url{dryanguasr@unal.edu.co}, \url{cauribe@mit.edu}, }
}

\begin{document}
\addtolength{\abovedisplayskip}{-.05cm}
\addtolength{\belowdisplayskip}{-.05cm}
\addtolength{\textfloatsep}{-.5cm}

\maketitle
\thispagestyle{empty}
\pagestyle{empty}

\begin{abstract}
We consider the model of cooperative learning via distributed non-Bayesian learning, where a network of agents tries to jointly agree on a hypothesis that best described a sequence of locally available observations. Building upon recently proposed weak communication network models, we propose a robust cooperative learning rule that allows asynchronous communications, message delays, unpredictable message losses, and directed communication among nodes. We show that our proposed learning dynamics guarantee that all agents in the network will have an asymptotic exponential decay of their beliefs on the wrong hypothesis, indicating that the beliefs of all agents will concentrate on the optimal hypotheses. Numerical experiments provide evidence on a number of network setups.   

\end{abstract}
\theoremseparator{.}

\newtheorem{proposition}{Proposition}
\newtheorem{theorem}{Theorem}
\newtheorem*{theoremnonumber}{Theorem}
\newtheorem{corollary}{Corollary}
\newtheorem{lemma}{Lemma}
\newtheorem{Fact}{Fact}
\newtheorem{remark}{Remark}
\newtheorem{assumption}{Assumption}
\newtheorem{definition}{Definition}

\newcommand{\eqdef}{\vcentcolon=}
\newcommand{\beq}{\begin{equation}}
\newcommand{\eeq}{\end{equation}}
\newcommand{\ie}{i.e., }

\section{Introduction}
Distributed inference has gained increasing attention in recent years due to the numerous applications in machine learning, sensor networks, decentralized control, and distributed signal processing. Among distributed inference models, non-Bayesian social learning has emerged as an essential approach to deal with decentralized heterogeneous learning over networks~\cite{Ali:12,Ali:18}. Non-Bayesian learning exhibits strong theoretical performance and allows large classes of sensing modalities and communication constraints. The non-Bayesian learning model assumes that the network of agents tries to agree on a set of beliefs about the state of the world that best describes a sequence of local observations from a finite set of possible states~\cite{Ali:12,Ali:18,Lalitha:18,shahrampour2013,Nedic:16tutorial}. Agents cooperate by exchanging and updating beliefs with their neighbors (in the network). Updating happens by aggregating information via some fusion rule. Moreover, the interaction between the agents over the network is usually modeled as a graph that defines and mediates such communication. 

Following the seminal work on Bayesian~\cite{Ace:11} and non-Bayesian \cite{Ali:12} social learning, several distributed learning rules, and their corresponding theoretical guarantees have been proposed in the literature. These rules consists of two main steps:  aggregation using a particular weighted (geometric) average of beliefs~\cite{Tsi:84,Ali:03,NeOl:15}; followed by Bayesian update of the aggregated beliefs \cite{Ace:11,Kan:13}. A number of variations of non-Bayesian social learning have been proposed, including distributed algorithms under different connectivity conditions such as weighted arithmetic averages~\cite{Ali:12, shahrampour2013}, geometric averages~\cite{RMJ2014, Lalitha:18}, constant elasticity of substitution model~\cite{Ali:18}, and minimum operators~\cite{sundaram2019new, MRS2019_CDC}. These learning rules have been applied to undirected/directed graphs, time-varying graphs~\cite{NOU2015,NOU2017}, weakly-connected graphs~\cite{ SYS2018}, agents with increasing self-confidence~\cite{UJ2019}, compact hypothesis sets~\cite{NOU2017_compact}, uncertain models~\cite{hare2020non}, and under adversarial attacks~\cite{BT2018, su2019defending, VSV2019, HULJ2019}.          

Previous works individually tackle harsh network conditions, such as asynchronous updates, delays, or message losses. An early attempt to deal with asynchronous updates is presented in~\cite{feldman2014reaching}, where agents are not required to send their beliefs every iteration-time. However, agreement of the beliefs depends on the network structure, and it is shown that there exist networks for which consensus is unlikely or converges to incorrect beliefs. In~\cite{Nithin2016asynchronous}, the effect of a finite set of simultaneous agents that suffer a crash fault is analyzed. It is assumed that up to a finite number of agents cease operating during any execution, and then it is proven the convergence of the network of agents to the true hypothesis. Scenarios where the message is delayed nor loss or that the agents are asleep not in fault are not considered. Additionally, we will see that network-independent learning rates can be achieved. Adversarial attacks have also been previously considered \cite{sundaram2019new,su2019defending}, but we will assume all agents are collaborative. 

In this paper, we build upon recently available results in distributed optimization considering asynchrony, delays, and message losses \cite{Wu:2017,Tian2018asy,hadjicostis2015robust,hadjicostis2018distributed}, and introduce a cooperative distributed non-Bayesian learning algorithm with robust performance guarantees under such harsh communication network conditions. In particular, we extend the recent proposed Robust Asynchronous Push-Sum (RAPS) consensus algorithm~\cite{JMLR:21} to the distributed learning setup.

\textit{The main contribution of this paper is threefold:}
\begin{itemize}
    \item We introduce a robust distributed non-Bayesian cooperative learning algorithm that considers asynchronous updates, communication delays, and unpredictable messages losses over a directed communication graph. A network of agents tries to jointly agree on a hypothesis that best describes a sequence of locally and asynchronously available observations.
    \item We show that our proposed learning dynamics guarantee that all agents in the network will have an asymptotic exponential decay of their beliefs on the wrong hypothesis, indicating that the beliefs of all agents will concentrate on the optimal hypotheses.
    \item We present numerical experiments that provide evidence of the proposed algorithm's performance on a number of network setups.
\end{itemize}

This paper is organized as follows. Section~\ref{sec:problem} introduces the distributed non-Bayesian problem, the communication network model, and the proposed algorithm and main asymptotic exponential convergence result. Section~\ref{sec:analysis} recalls the robust asynchronous push-sum algorithm (RAPS)~\cite{JMLR:21}, and proofs our main belief concentration result. Section~\ref{sec:numerics} presents numerical results for the proposed algorithm for three different network topologies and communication failures. Finally, conclusions and future work are shown in Section~\ref{sec:conclusions}.

\noindent \textbf{Notation:} We denote random variables as capital letters, e.g., $X$, and use its corresponding lower case for its realizations, e.g., $x$. Node indices are usually denoted by the letters $i$ and $j$, time indices or iterations are denoted by $t$ or $k$. Bold letters will usually denote concatenation or stacking of vectors. We denote $[A]_{ij}$ as the entry of the $A$ matrix at its $i$-th row and $j$-th column. $\circ$ denotes entry wise-product. Vector are assumed to be column vectors and $x^\top$ denotes the transpose of the vector $x$. We write $\boldsymbol{1}$ as the all-ones vectors with appropriate dimension. For a sequence of matrices $\{A_k\}_{k\geq 0}$, we let $A(k_f:k_i) \eqdef A_{k_f}\cdots A_{k_i+1}A_{k_i}$ for all $k_f\geq k_i\geq 0$.

\section{Cooperative Learning: Problem statement, Network model, and Results}\label{sec:problem}
In this section, we present the distributed non-Bayesian learning problem. Furthermore, we state the general network model we will consider. Additionally, we present the proposed learning dynamics and the asymptotic convergence results.

\subsection{Problem Statement}

Consider a network of $n$ agents on a set of nodes ${V}=\{1,2, \ldots, n\}$ observing realizations of a finite, stationary, independent, identically distributed random processes $\{X_k\}_{k\geq 1}$ where $X_k^i \sim P^i$ at each iteration time $k$ with unknown distribution $P^i$. Additionally, all agents have a shared finite set of hypotheses  $\Theta = \{\theta_1, \theta_2, \ldots, \theta_m \}$, from which each agent $i\in V$ defines a local family of distributions $\mathcal{P}^i = \{P^i_{\theta} \mid \theta \in \Theta\}$. We will assume the technical condition that each element in family of distributions $\mathcal{P}^i$ is absolutely continuous with respect to $P^i$. We denote $\mathcal{N}_i^{+}$ and $\mathcal{N}_i^{-}$   as the set of out-neighbors and in-neighbors of an agent $i$.

The objective of the network of agents is to agree on a parameter $\theta^* \in \Theta$ such that the joint distribution $\prod P^i_{\theta^*}$ is closest (in a statistical sense to be defined later) to $\prod P^i$.

Formally, the group of agents tries to solve jointly
\begin{align}\label{main:problem}
    \min_{\theta \in \Theta} F(\theta) \triangleq \sum_{i\in V} D_{KL}(P^i \| P^i_\theta),
\end{align}
where $D_{KL}(P\|Q)$ is the Kullback-Leibler divergence between the distributions $P$ and $Q$. Importantly, note that each of the agents only knows its local family of distributions $\mathcal{P}^i$, the true distribution of their local observations $P^i$ is unknown, yet accessible via local observations. Thus, in order to solve \eqref{main:problem}, cooperation is needed. Moreover, note that we have not assumed that the minimizer of~\eqref{main:problem} is unique. We will generally denote the set of minimizers as $\Theta^*$.

\begin{remark}
Note that we have not assumed that there is a $\theta \in \Theta$ for which $D_{KL}(P^i \| P^i_\theta)=0$. Thus, the true distribution of the observations might not be in the hypotheses set; such a scenario is usually referred to as \textit{misspecified models}.
\end{remark}

The confidence each agent has on each of the hypotheses in~$\Theta$ is represented by a \textit{belief} vector, denoted as $\mu_i^\theta(k)$, which indicates the belief that an agent $i\in V$ has about a hypothesis $\theta\in\Theta$ at certain time instant $k$. A value of $\mu_i^\theta(k)= 1$ indicates certainty that the minimizer of~\eqref{main:problem} is $\theta$, whereas $\mu_i^\theta(k) = 0$ indicates certainty that it is not.

\subsection{Network Model}

Agents cooperate by communicating their beliefs at each time instant. Such communication is mediated by a network, modeled as a graph $\mathcal{G} = \{V, E\}$. Following the assumptions from~\cite[Assumption 1]{JMLR:21}, we consider networks for which there might be communication delays in the links, asynchronous node activation, and link failures. 
` 
\begin{assumption}[Assumption 1 in~\cite{JMLR:21}] \label{assum:network}
Suppose:
\begin{itemize}
    \item[(a)] Graph $\mathcal{G}$ is strongly connected and does not have self-loops.
    \item[(b)] The delays on each link are bounded above by some $L_{del}\geq 1$.
    \item[(c)] Every agent wakes up and performs updates at least every $L_u \geq 1$ iterations.
    \item[(d)] Each link fails at most $L_f\geq 1$ consecutive times.
    \item[(e)] Messages arrive in the order of time they were sent. In other words, if message is sent from node $i$ to node $j$ at time $k_1$ and $k_2$, respectively, with effective delays $d_1$ and $d_2$, then $k_1+d_1 < k_2+d_2$.
\end{itemize}
\end{assumption}

Assumption~\ref{assum:network} is, to the best of the author's knowledge, the weakest assumptions available in the literature for cooperative learning. It allows finite communication delays between agents and possible link failures, and asynchronous activation of the nodes. However, we impose the useful assumption that messages will arrive in the order they were sent. As pointed out in~\cite{JMLR:21}, the main consequence of Assumption~\ref{assum:network} is that effective delays will be bounded by $L_{del}+L_u-1$. Moreover, if $(j,i)\in E$, node $j$ will receive a message from node $i$ successfully, at least once every $L_s = L_u(L_f+1) + (L_{del}+L_u-1)$. The reader is referred to~\cite{JMLR:21} for more details about this network model.  

We will also need an additional assumption that guarantees the log-likelihood ration of the distributions in the hypotheses set is bounded, both from above and bellow.
\begin{assumption}\label{assum:bound}
    There exists a $\beta>0$ such that $P^i_\theta(\cdot) \geq \beta$ for all $\theta \in\Theta$ for all $i\in V$.
\end{assumption}

\subsection{Proposed Learning Algorithm and Main Result}

Next, we state Algorithm~\ref{alg:main}, our proposed cooperative learning algorithm, and state our main results.

\begin{algorithm}[t]
    \caption{Robust Asynchronous Push-Sum Distributed Non-Bayesian Learning}
    \begin{algorithmic}[1]
     \State{\textbf{Initialize:} $y_i(0)=1$, $\phi_i^y(0)=0$, $\phi_{i,\theta}^{\mu}(0)=1$, $\forall i \in V$, and $\rho_{ij}^y(0)=0$, $\kappa_{ij}(0)=0, \forall (i,j) \in \mathcal{E}$}
   \State{Set initial beliefs as uniform for all agents.}
   \For{ $k=0,1,2, \ldots,$ for every node $i$:}   
        \If{Node $i$ wakes up}
        \Statex{\textbf{1. Processing and broadcasting local information}}
        \State{$\kappa_i \gets k$, $\phi_{i}^{y} \gets \phi_{i}^y+{y_i}/({d_i^+ +1})$\;}
        \State{$\phi_{(i,\theta)}^{\mu} \gets \phi_{(i,\theta)}^{\mu}\Big(\mu_i^{\theta}\Big)^{{y_i}/({d_i^+ +1})}$ }
        \State{Node $i$ broadcasts $(\phi_i^y,\phi_{(j,\theta)}^{\mu}, \kappa_i)$ to $\mathcal{N}_i^+$.}
        \Statex{\textbf{2. Processing received messages}}
        \For{$(\phi_j^y, \phi_{(i,\theta)}^{\mu}, \kappa_j')$ in the inbox} 
            \If{$\kappa_j' > \kappa_{ij}$}
                \State{$\rho_{ij}^{*y} \gets \phi_j^y$, $\rho_{ij|\theta}^{*\mu} \gets \phi_{(j,\theta)}^{\mu}$, $\kappa_{ij} \gets \kappa_j'$}
            \EndIf
        \EndFor
        \Statex{\textbf{3. Updating beliefs and local information}}
                \State{$\hat{y}_i \gets \frac{y_i}{d_i^+ +1}+\sum\limits_{j \in \mathcal{N}_i^-}(\rho_{ij}^{*y}-\rho_{ij}^y)$\;}
                \State{$\mu_i^{\theta} \gets \frac{1}{Z_i}\Big(\Big(\mu_i^{\theta}\Big)^{\frac{y_i}{d_i^+ +1}}\prod\limits_{j\in\mathcal{N}_i^-} \Big(\frac{\rho_{ij\mid \theta}^{*\mu}}{\rho_{ij\mid \theta}^{\mu}}\Big) P^i_\theta(x^i_{k+1})\Big)^{\frac{1}{\hat{y}_i}}$\;}
                \Statex{$Z_i$ is a normalization constant.}
                \State{$y_i \gets \hat{y}_i$, $\rho_{ij}^y \gets \rho_{ij}^{*y}$, $\rho_{ij \mid \theta}^{\mu} \gets \rho_{ij\mid \theta}^{*\mu}$\;}
        \EndIf 
     \EndFor
    \end{algorithmic}
    \label{alg:main}
\end{algorithm}

In Algorithm~\ref{alg:main}, each awake node executes three main states at every iteration. Initially, local variables are updated with the most recent information about outgoing neighbors for each possible hypothesis. This local processing step is concluded by broadcasting auxiliary variables and time-stamps to its available out-neighbors at that particular time. Then, each agent modes on processing the messages it might have arrived from its in-neighbors while not awake. Each agent first checks time-stamps for each of the messages and updates the stored neighbor information if newer information is available. Finally, the node updates its beliefs with the most recent information from its neighbors, and its local observation of the random variable $X_k^i$, and goes to sleep mode again. This process repeats at each iteration.

Now, we are ready to state the main result of this paper. Specifically, we show that the learning dynamics proposed in Algorithm~\ref{alg:main} guarantees that the beliefs of all agents will concentrate in the set of minimizers of $F(\theta)$, denoted as $\Theta^*$. The proof of this theorem will be shown in Section~\ref{sec:analysis}.

\begin{theorem}[Main Result]\label{thm:main}
Let Assumptions~\ref{assum:network} and~\ref{assum:bound} hold. Then, the output of Algorithm~\ref{alg:main} has the following property:
\begin{equation}
    \lim_{k \to \infty} \frac{1}{k}\log\frac{\mu^i_{\theta_v}(k)}{\mu^i_{\theta_w}(k)} \leq - \frac{1}{n} \min_{\theta \notin \Theta^*}\left( F(\theta) -F(\theta^*) \right)
\end{equation}
almost surely for all $\theta_v \notin \Theta ^*$, and $\theta_w \in \Theta ^*$, and $i\in V$.
\end{theorem}

Theorem~\ref{thm:main} states that for all non-optimal hypothesis, the beliefs will decay asymptotically exponentially fast. Moreover, the rate at which the beliefs will decay asymptotically is upper bounded by the averaged optimality gap of the second-best hypothesis, i.e., $F(\theta) -F(\theta^*)$. If the model is well specified, we can expect $F(\theta^*)=0$, which will make the concentration rate large. However, the \textit{closer} (in the sense of Kullback-Leibler) the optimal and the closest suboptimal hypothesis are, the slower the concentration will happen. Additionally, the concentration rate is network-independent, in the sense that as long Assumption~\ref{assum:network} holds, the concentration rate will not depend on the specific network topology. Such asymptotic network-independence additionally implies that the effects of the network topology are transient. 

\begin{remark}
Note that the result in Theorem~\ref{thm:main} implies that $\lim_{k \to \infty} \mu^i_k(\theta) =0 $ almost surely for all $\theta \notin \Theta^*$.
\end{remark}

In the next section, we will analyze Algorithm~\ref{alg:main} and provide formal proof for Theorem~\ref{thm:main}.


\section{Asymptotic Exponential Convergence Analysis}\label{sec:analysis}


In this section, we prove our main result in Theorem~\ref{thm:main}. To do so, we first recall an auxiliary result that will allow us to show a consensus under Assumption~\ref{assum:network}. We proceed to show the asymptotic exponential concentration of the beliefs for all agents.

\subsection{Robust Asynchronous Push-Sum (RAPS)}

In this subsection, we briefly recall the main result from~\cite{JMLR:21} about the convergence of a robust asynchronous version of the push-sum algorithm, initially the consensus algorithm over directed graphs was proposed in~\cite{tsianos2012push}. First, let us recall the RAPS Algorithm proposed in~\cite{JMLR:21}. 

\begin{algorithm}[t]
    \caption{Robust Asynchronous Push-Sum (RAPS)}
    \begin{algorithmic}[1]
     \State{\textbf{Initialize:} $y_i(0)=1$, $\phi_i^x(0)=0$, $\phi_{i}^y(0)=0$, $\forall i \in V$, and $\rho_{ij}^x(0)=0$, $\rho_{ij}^y(0)=0$, $\kappa_{ij}(0)=0, \forall (i,j) \in \mathcal{E}$}
   \For{ $k=0,1,2, \ldots,$ for every node $i$:}   
        \If{Node $i$ wakes up}
        \State{$\kappa_i \gets k$, $\phi_{i}^{y} \gets \phi_{i}^y+\frac{y_i}{d_i^+ +1}$, $\phi_{i}^{x} \gets \phi_{i}^x+\frac{x_i}{d_i^+ +1}$\;}
        \State{Node $i$ broadcasts $(\phi_i^y,\phi_i^x, \kappa_i)$ to $\mathcal{N}_i^+$.}
        \For{$(\phi_j^y, \phi_i^x, \kappa_j')$ in the inbox} 
            \If{$\kappa_j' > \kappa_{ij}$}
                \State{$\rho_{ij}^{*y} \gets \phi_j^y$, $\rho_{ij|\theta}^{*\mu} \gets \phi_{(j,\theta)}^{\mu}$, $\kappa_{ij} \gets \kappa_j'$}
            \EndIf
        \EndFor
                \State{$y_i \gets \frac{y_i}{d_i^+ +1}+\sum\limits_{j \in \mathcal{N}_i^-}(\rho_{ij}^{*y}-\rho_{ij}^y)$\;}
                \State{$x_i \gets \frac{x_i}{d_i^+ +1} +\sum\limits_{j \in \mathcal{N}_i^-}(\rho_{ij}^{*x}-\rho_{ij}^x)$\;}
                \State{$\rho_{ij}^y \gets \rho_{ij}^{*y}$, $\rho_{ij}^x \gets \rho_{ij}^{*x}$, $z_i \gets \frac{x_i}{y_i}$}
        \EndIf 
     \EndFor
    \end{algorithmic}
    \label{alg:raps}
\end{algorithm}

Algorithm~\ref{alg:raps} was shown to guarantee consensus agreement on the average initial values of the variable $x_i$ among all nodes, for a network model for which Assumption~\ref{assum:network} holds. This result is stated in the next theorem.

\begin{theorem}[Theorem 6 in~\cite{JMLR:21}]\label{thm:raps}
Suppose Assumption~\ref{assum:network} holds. Then, RAPS converges exponentially to the initial mean of agent values, i.e.,
\begin{align*}
    \Big| z_i(k) - \frac{1}{n}\sum_{i=1}^n x_i(0) \Big| \leq \delta \lambda^k \|\mathbf{x}(0)\|_1,
\end{align*}
where $\delta\eqdef 1/(1-n \alpha^6)$, $\lambda \eqdef (1-n\alpha^6)^{1/(2nL_s)}$, and $\alpha\eqdef 1/n^{nL_s}$.
\end{theorem}

The key insight into the convergence analysis of the RAPS algorithm (Algorithm~\ref{alg:raps}) is that, although not evident at first sight, one can write its dynamics of the variables $z_i$ as the ratio of two linear processes, with appropriately defined mixing matrices. In particular, the authors in~\cite{JMLR:21} showed that the iterates in Algorithm~\ref{alg:raps} are equivalent to a couple of processes
\begin{align}\label{eq:linear_raps}
    \boldsymbol{\chi}(k{+}1) {=}\mathbf{M}(k) \boldsymbol{\chi}(k), \ \text{and}\ \boldsymbol{\phi}(k{+}1) {=} \mathbf{M}(k) \boldsymbol{\phi}(k),
\end{align}
where $\mathbf{M}(k)(k)\in \mathbb{R}^{(n+m')\times (n+m')}$ is a sequence of appropriately defined column-stochastic matrices~\cite[Lemma 5]{JMLR:21}. More importantly, the first first $n$ elements of the vectors $\boldsymbol{\chi}(k)$ and $\boldsymbol{\phi}(k)$ are $\mathbf{x}(k)$, and $\mathbf{y}(k)$, respectively.

\subsection{Robust Asynchronous Push-Sum Distributed Non-Bayesian Learning}

Our general strategy to show the exponential asymptotic convergence of the proposed method is to initially show that the dynamics of Algorithm~\ref{alg:main} will follow a perturbed version of the linear dynamics~\eqref{eq:linear_raps}. This is going to be our first goal. Before presenting such a relation, we define a couple of auxiliary variables that simplify notation and state the result of the linear dynamics.

\begin{itemize}
    \item $y_{ij}^l$: The variable $y_{ij}^l$ for $(i,j)\in E$ and $1\leq l \leq L_d$ indicates the information that was sent from node $i$ to node $j$ and will arrive with an effective delay of $l$. Effectively, we are introducing ``virtual nodes'' as buffers inducing delays in the communications, and the variables $y_{ij}^l$ will be the values held by those virtual nodes. The idea of virtual nodes was originally proposed in~\cite{hadjicostis2015robust}. See~\cite{hadjicostis2018distributed}, for a comprehensive study of the virtual nodes approach via exchanging accumulated sums.
    \item $\tau_i(k)$: The variable $\tau_i(k)$ for $i \in V$ is an indicator variable that equals $1$ if node $i$ has waken up at time $k$, and equals $0$ otherwise. \item $\tau_{ij}^l(k)$: The variable $\tau_{ij}^l(k)$ for $(i,j)\in E$ and $1\leq l \leq L_d$ is $1$ if $\tau_i(k)=1$ and the messages sent from node $i$ to node $j$ at time $k$ will arrive after an effective delay of $l$.
    \item $u_{ij}^y(k)$: The variable $u_{ij}^y(k)$ is defined recursively as
    \resizebox{1.05\linewidth}{!}{
  \begin{minipage}{\linewidth}
  \begin{align*}
\hspace{-0.91cm} u_{ij}^y(k+1) \eqdef \Big( 1 {-} \sum_{l=1}^{L_d}\tau^l_{ij}(k)\Big) \Big(u_{ij}^y(k) {+ }\phi_i^x(k{+}1) {-} \phi_i^x(k)    \Big),
\end{align*}
  \end{minipage}
}
    with initial value $u_{ij}^y(0)\eqdef 0$.
    \item $\varphi_i^{\bar{\theta}}(k)$: Defined as $\varphi_i^{\bar{\theta}}(k)=  y_i(k)\log ({\mu_i^{\theta_v}(k)}/{\mu_i^{\theta_w}(k)})$ for $\bar{\theta}=\{\theta_v, \theta_w\}$, where  $\theta_v \notin \Theta^*$, and $\theta_w \in \Theta^*$.
    \item $u_{ij\mid\theta}^\mu(k)$: The variable $u_{ij}^y(k)$ is defined recursively as
    \resizebox{1.05\linewidth}{!}{
  \begin{minipage}{\linewidth}
  \begin{align*}
\hspace{-0.1cm}  u_{ij|\theta}^{\mu}(k+1)=\left( u_{ij|\theta}^{\mu}(k)\frac{\phi_{i,\theta}^\mu (k+1)}{\phi_{i,\theta}^\mu (k)}\right)^{\left(1-\sum_{l=1}^{L_d}\tau_{ij}^l(k)\right)},
\end{align*}
  \end{minipage}
}
    with initial value $u_{ij\mid \theta }^\mu(0)\eqdef 1$.
    \item $\upsilon^{\bar\theta}_{ij}(k)$: Defined as $\upsilon^{\bar\theta}_{ij}(k) = \log (u_{ij\mid\theta_v}^\mu(k)/u_{ij\mid\theta_w}^\mu(k))$
    \item $\mu_{ij|\theta}^{l}(k)$: The variable $\mu_{ij|\theta}^{l}(k)$ is defined as
    \resizebox{1.00\linewidth}{!}{
  \begin{minipage}{\linewidth}
  \begin{align*}
\hspace{-0.1cm} \mu_{ij|\theta}^{l}(k+1)=\left(u_{ij\mid \theta}^{\mu}(k)\big(\mu_i^{\theta}(k)\big)^{y_i(k)/(d_i^+ +1)}\right)^{\tau_{ij}^{l}(k)}\mu_{ij\mid \theta}^{l+1}(k),
\end{align*}
  \end{minipage}
}
where
\begin{align*}
      \mu_{ij|\theta}^{L_d}(k+1)=\left(u_{ij}^{\mu}(k)\big(\mu_i^{\theta}(k)\big)^{y_i(k)/(d_i^+ +1)}\right)^{\tau_{ij}^{L_d}(k)}.
\end{align*}
 \item $\varphi_{ij}^{\bar{\theta},l}(k)$: Defined as $\varphi_{ij}^{\bar{\theta},l}(k)=  \log ({\mu_{ij\mid\theta_v}^l(k)}/{\mu_{ij\mid\theta_w}^l(k)})$ for $\bar{\theta}=\{\theta_v, \theta_w\}$, where  $\theta_v \notin \Theta^*$, and $\theta_w \in \Theta^*$.
\end{itemize}

The previous definitions might appear out of context or hard to parse at this time. However, such variables will significantly reduce notation and analysis in the following.

\begin{lemma}\label{lemma:linear}
For every pair $\bar{\theta}=\{\theta_v, \theta_w\}$, where $\theta_v\notin \Theta^*$ and $\theta_w \in \Theta^*$. the iterates generated by Algorithm~\ref{alg:main} are equivalent to the following pair of linear processes:
\begin{align}
    \mathbf{Y}(k+1) &=\mathbf{M}(k)\mathbf{Y}(k) \label{eq:linear_y}, \\
    \bm{\psi}^{\bar\theta}(k+1)&=\mathbf{M}(k)\bm{\psi}^{\bar\theta}(k)+ \bm{L}^{\bar\theta}(k+1),\label{eq:linear_phi}
\end{align}
where the matrix $\mathbf{M}(k)$ is the same matrix as in~\eqref{eq:linear_raps}, and
\begin{align*}
    \mathbf{Y}(k) & \eqdef [\mathbf{y}(k)^{\top}, \mathbf{y}^1(k)^{\top}, \ldots, \mathbf{y}^{L_d}(k)^{\top}, \mathbf{u}^y(k)^{\top}]^{\top},\\
    \bm{\psi}^{\bar\theta}(k) & \eqdef [\boldsymbol{\varphi}^{\bar{\theta}}(k)^{\top}, \boldsymbol{\varphi}^{\bar{\theta},1}(k)^{\top}, \ldots, \boldsymbol{\varphi}^{\bar{\theta},L_d}(k)^{\top}, \boldsymbol{\upsilon}^{\bar\theta} (k)^{\top}]^{\top},
    \end{align*}
where $\mathbf{y}(k)$ stacks all $y_i(k)$, $\mathbf{y}^l(k)$ stacks all $y_{ij}
^l$, and $\mathbf{u}^y(k)$ stacks all $u_{ij}^y(k)$. Similarly, $\boldsymbol{\varphi}^{\bar{\theta},l}(k)$ stacks all $\varphi_i^{\bar{\theta}}(k)$, $\boldsymbol{\varphi}^{\bar{\theta},l}(k)$ stacks all $\varphi_{ij}^{\bar{\theta},l}(k)$, and $\boldsymbol{\upsilon}^{\bar\theta} (k)$ stacks all  $\upsilon^{\bar\theta}_{ij}(k)$. Finally, the first $n$ entries of the vector $\bm{L}^{\bar\theta}(k)$ stacks all $\log(P^i_{\theta_v}(X_{k}^i)/P^i_{\theta_w}(X_{k}^i))$, and zeros in all other entries.
\end{lemma}

\begin{proof}
The relation in~\eqref{eq:linear_y} follows immediately from~\cite[Lemmas 1,2,3,4]{JMLR:21} since the evolution of the variables $y_i$ is the same for both the RAPS Algorithm and Algorithm~\ref{alg:main}. Unfortunately this is not the case for~\eqref{eq:linear_phi}, which we will show next.

Lets consider some arbitrary time $k\geq 0$. Thus, from Line~$7$ in Algorithm~\ref{alg:main} it follows that
\begin{align}\label{eq:phi_u}
 \hspace{-0.3cm}   \log\frac{\phi^{\mu}_{i,\theta_v}(k{+}1)}{\phi^{\mu}_{i,\theta_w}(k{+}1)} {=} \log\frac{\phi^{\mu}_{i,\theta_v}(k)}{\phi^{\mu}_{i,\theta_w}(k)}{+}\frac{\tau_i(k)}{d_i^{+}{+}1}y_i(k)\log\frac{\mu_i^{\theta_v}(k)}{\mu_i^{\theta_w}(k)},
\end{align}
and from Line~$17$ in Algorithm~\ref{alg:main} we get
\begin{align}\label{eq:log_process}
    &\log \frac{\mu^{\theta_v}_i(k+1)}{\mu^{\theta_w}_i(k+1)}  \nonumber \\
    & = \frac{1}{y_i(k+1)} \Bigg(y_i(k) \Big(1-\tau_i(k)+\frac{\tau_i(k)}{d_i^{+}+1} \Big)\log \frac{\mu^{\theta_v}_i(k)}{\mu^{\theta_w}_i(k)} + \nonumber \\
    & \qquad  +\sum_{j \in\mathcal{N}_i^{-}} \Big(\log \frac{\rho_{ij \mid \theta_v}^{\mu} (k+1)}{\rho_{ij \mid \theta_w}^{\mu}(k+1) } - \log \frac{\rho_{ij \mid \theta_v}^{\mu} (k)}{\rho_{ij \mid \theta_w}^{\mu}(k) }  \Big)  \nonumber \\
    & \qquad + \log\frac{P_{\theta_v}^i(X^i_{k+1})}{P_{\theta_w}^i(X^i_{k+1})}  \Bigg),
\end{align}
where $y_i({k+1}) = y_i(k)/(d_i^++1) + \sum_{j \in\mathcal{N}_i^{-}}(\rho_{ij}^{*y}-\rho_{ij}^{y})$.

Now, from~\cite[Lemma 4]{JMLR:21}, we have that for all $k\geq 0$ and $(i,j) \in E$,
\begin{align*}
    \rho_{ji\mid \theta}^{\mu}(k+1) & = \mu^1_{ij\mid \theta}(k) \rho_{ji\mid \theta}^{\mu}(k).
\end{align*}
Thus, we can write~\eqref{eq:log_process} as
\begin{align}\label{eq:aux_4}
    &\varphi^{\bar\theta}_{i}(k+1)  = \Big(1-\tau_i(k)+\frac{\tau_i(k)}{d_i^{+}+1} \Big)\varphi^{\bar\theta}_{i}(k) + \nonumber \\
    & \qquad  +\sum_{j \in\mathcal{N}_i^{-}} \log \frac{\mu^1_{ij\mid \theta_v}(k)}{\mu^1_{ij\mid \theta_w}(k)}  + \log\frac{P_{\theta_v}^i(X^i_{k+1})}{P_{\theta_w}^i(X^i_{k+1})}. 
\end{align}

Additionally, from the definition of $u_{ij|\theta_v}^{\mu}(k)$ and~\eqref{eq:phi_u}, it follows that 
\begin{align}
    & \log \frac{u_{ij|\theta_v}^{\mu}(k+1)}{u_{ij|\theta_w}^{\mu}(k+1)}= \Big(1-\sum_{l=1}^{L_d}\tau_{ij}^k(k)\Big) \Bigg( \log \frac{u_{ij|\theta_v}^{\mu}(k)}{u_{ij|\theta_w}^{\mu}(k)} +  \nonumber \\
    & \qquad + \log \frac{\phi^\mu_{i\mid \theta_v}(k+1)}{\phi^\mu_{i\mid \theta_w}(k+1)} - \log \frac{\phi^\mu_{i\mid \theta_v}(k)}{\phi^\mu_{i\mid \theta_w}(k)} \Bigg) \nonumber\\
    & =  \Big(1-\sum_{l=1}^{L_d}\tau_{ij}^k(k)\Big) \Bigg( \log \frac{u_{ij|\theta_v}^{\mu}(k)}{u_{ij|\theta_w}^{\mu}(k)} + \nonumber\\
    & \qquad + \frac{\tau_i(k)}{d_i^+ +1} y_i(k) \log\frac{\mu_k^{\theta_v}(k)}{\mu_k^{\theta_w}(k)} \Bigg) \nonumber \\
    & \boldsymbol{\upsilon}_{ij}^{\bar{\theta}}(k{+}1) {=}  \Big(1-\sum_{l=1}^{L_d}\tau_{ij}^k(k)\Big) \Bigg( \boldsymbol{\upsilon}_{ij}^{\bar{\theta}}(k) {+} \frac{\tau_i(k)}{d_i^+ {+}1} \varphi_i^{\bar \theta }(k) \Bigg). \label{eq:aux_3}
\end{align}

Finally, from the definition of $\mu_{ij|\theta}^{l}(k+1)$, it follows that
\begin{align}
    &\log \frac{\mu_{ij|\theta_v}^{l}(k{+}1)}{\mu_{ij|\theta_w}^{l}(k{+}1)} = \tau_{ij}^l(k)\Big(\log\frac{u_{ij\mid\theta_v}(k)}{u_{ij\mid\theta_w}(k)} {+} \nonumber \\ 
    & \qquad + \frac{y_i(k)}{d_i^+ {+}1} \log \frac{\mu_i^{\theta_v}(k)}{\mu_i^{\theta_w}(k)}\Big) {+} 
     \log \frac{\mu_{ij\mid\theta_v}^{l+1}(k)}{\mu_{ij\mid\theta_w}^{l+1}(k)},  \nonumber \\
    &  \hspace{-0.2cm}\varphi^{\bar\theta , l}_{ij}(k{+}1)  {=} \tau_{ij}^l(k) \Big( \upsilon^{\bar\theta}_{ij}(k) {+} \frac{1}{d_i^+ {+}1 } \phi^{\bar \theta}_i(k)\Big) {+} \varphi^{\bar\theta,l+1}_{ij}(k), \label{eq:aux_2}
\end{align}
and similarly,
\begin{align}\label{eq:aux_1}
     \varphi^{\bar\theta , L_d}_{ij}(k+1)  = \tau_{ij}^{L_d}(k) \Big( \upsilon^{\bar\theta}_{ij}(k) + \frac{1}{d_i^+ +1 } \phi^{\bar \theta}_i(k)\Big) .
\end{align}

The desired result follow by concatenating~\eqref{eq:aux_1}, \eqref{eq:aux_2}, \eqref{eq:aux_3}, and \eqref{eq:aux_4}.

\end{proof}

\subsection{Proof of Asymptotic Exponential Concentration of Beliefs}

Before proving our main result, we propose a slight reformulation of Theorem~\ref{thm:raps} that will allow us to prove our main result about the asymptotic exponential concentration of the beliefs of all agents around the optimal hypotheses set.

Note from Theorem~\ref{thm:raps} and~\eqref{eq:linear_raps} we can define a vector $\boldsymbol{z}(k) \eqdef \boldsymbol{\chi}(k)\circ \boldsymbol{\phi}^{-}(k)$, where ${\phi}^{-}_i(k) =1/{\phi}_i(k) $ if ${\phi}_i(k)\neq 0 $ and ${\phi}_i(k)=0$ otherwise. Thus, $\boldsymbol{z}(k+1) = \boldsymbol{P}(k)\boldsymbol{z}(k)$, for $\boldsymbol{P}(k)\eqdef \text{diag}(\boldsymbol{\phi}^{-}(k+1))\boldsymbol{M}(k)\boldsymbol{\phi}(k)$. Thus, equivalently, we can write
\begin{align}\label{eq:aux_thm}
    \Big| [\boldsymbol{P}(k:t)]_{ij} - \frac{1}{n}  \Big| \leq \delta \lambda^{k-t}, \ \ \forall k \geq t \geq 0.
\end{align}

We are now ready to prove Theorem~\ref{thm:main}.

\begin{proof}[Theorem~\ref{thm:main}] Initially, define the process
\begin{align*}
    \boldsymbol{\Phi}^{\bar\theta}(k)\eqdef \boldsymbol{\psi}^{\bar\theta}(k) \circ \boldsymbol{Y}^{-}(k),
\end{align*}
with $Y_i^{-}(k) = 1/Y_i(k)$ if $Y_i(k)\neq 0$, and  $Y_i^{-}(k)=0$ otherwise. Thus,
\begin{align*}
    &\boldsymbol{\Phi}^{\bar\theta}(k+1)  =  \boldsymbol{P}(k)\boldsymbol{\Phi}^{\bar\theta}(k)+\boldsymbol{L}^{\bar \theta}(k+1)\circ\boldsymbol{Y}^{-}(k+1)\\
    & = \boldsymbol{P}(k{:}0)\boldsymbol{\Phi}^{\bar\theta}(0) {+} \sum_{t=0}^{k}\boldsymbol{P}(k{:}t)\boldsymbol{L}^{\bar \theta}(t){+}\boldsymbol{L}^{\bar \theta}(k+1){\circ}\boldsymbol{Y}^{{-}}(k)
\end{align*}

The first term in the above expression is precisely equal to zero following our choice of initial beliefs being uniform. Now, lets focus on the second term, i.e., $\sum_{t=0}^{k}\boldsymbol{P}(k{:}t)\boldsymbol{L}^{\bar \theta}(t)$. If we add and subtract the matrix $\frac{1}{n}\boldsymbol{1}^\top \boldsymbol{1}$, we obtain,
\begin{align*}
    \sum_{t=0}^{k}\big(\boldsymbol{P}(k{:}t)-\frac{1}{n}\boldsymbol{1}^\top \boldsymbol{1}\big)\boldsymbol{L}^{\bar \theta}(t)) + \sum_{t=0}^{k}\frac{1}{n}\boldsymbol{1}^\top \boldsymbol{1}\boldsymbol{L}^{\bar \theta}(t),
\end{align*}
and now similarly lets add and subtract the vector $\boldsymbol{H^{\bar\theta}} = \mathbb{E}\boldsymbol{L}^{\bar \theta}(t)$, which is the expected value of the logarithmic ration of the likelihood functions. Thus, we have
\begin{align*}
    \frac{1}{n}\sum_{t=0}^{k}\sum_{l=1}^{n+m'}({L}_i^{\bar \theta}(t) -{H}_i^{\bar \theta}(t)) \boldsymbol{1}+\frac{1}{n}\sum_{t=0}^{k}\sum_{l=1}^{n+m'}{H}_i^{\bar \theta}(t) \boldsymbol{1}\\
    = \frac{1}{n}\sum_{t=0}^{k}\sum_{l=1}^{n+m'}({L}_i^{\bar \theta}(t) -{H}_i^{\bar \theta}(t)) \boldsymbol{1}- k\frac{1}{n}\sum_{i=1}^n D_{KL}(P^i_{\theta_w}\|P^i_{\theta_v}),
    \end{align*}
    where we have use the fact that ${H}_i = -D_{KL}(P^i_{\theta_w}\|P^i_{\theta_v})$.
    
As a final step, we show the limit of the process $\frac{1}{k}\boldsymbol{\Phi}^{\bar\theta}(k)$ as $k \to\infty$.
\begin{align*}
    \lim_{k\to \infty} \frac{1}{k}\boldsymbol{\Phi}^{\bar\theta}(k) {=} \underbrace{ \lim_{k\to \infty} \frac{1}{k} \sum_{t=0}^{k}\big(\boldsymbol{P}(k{:}t){-}\frac{1}{n}\boldsymbol{1}^\top \boldsymbol{1}\big)\boldsymbol{L}^{\bar \theta}(t))}_{\textcircled{1}} {+} \\
   + \underbrace{ \lim_{k\to \infty} \frac{1}{k} \frac{1}{n}\sum_{t=0}^{k}\sum_{l=1}^{n+m'}({L}_i^{\bar \theta}(t) {-}{H}_i^{\bar \theta}(t)) \boldsymbol{1}}_{\textcircled{2}}{-} \frac{1}{n}\sum_{i=1}^n D_{KL}(P^i_{\theta_w}\|P^i_{\theta_v}).
\end{align*}

The term $\textcircled{1}$ is equal to zero deterministically as an immediate consequence of~\eqref{eq:aux_thm} (c.f. Theorem~\ref{thm:raps}), and Assumption~\ref{assum:bound}. Moreover, the term $\textcircled{2}$ is equal to zero almost surely due to the Assumption~\ref{assum:bound}, which implies that the variance of ${L}_i^{\bar \theta}(k)$ is bounded for all $k\geq0$, independence of the sequence of observations, and a direct application of Kolmogorov's strong law of large numbers.

Therefore we can conclude
\begin{align*}
     \lim_{k\to \infty} \frac{1}{k}\boldsymbol{\Phi}^{\bar\theta}(k) {=}{-} \frac{1}{n}\sum_{i=1}^n D_{KL}(P^i_{\theta_w}\|P^i_{\theta_v}),
\end{align*}
and in particular
\begin{align*}
     \lim_{k\to \infty} \frac{1}{k} \log \frac{\mu^{\theta_v}_i(k)}{\mu^{\theta_w}_i(k)} {=}{-} \frac{1}{n}\sum_{i=1}^n D_{KL}(P^i_{\theta_w}\|P^i_{\theta_v}),
\end{align*}
and  the desired result follows.
\end{proof}

\section{Numerical Experiments}\label{sec:numerics}

In this section, we present numerical experiments to illustrate the behavior of Algorithm~\ref{alg:main} under the network model that allows asynchronous updates, delays, and link failures.

We take a network with $n = 4$ agents connected through three different base topologies: a path, a start, and a cycle graph. We simulate link failures by assigning a probability of failing, where the package is lost (if the number of consecutive fails exceeds $L_f=5$, the is forced to connect at least once). Moreover, agents have a certain probability of waking at each time instant $k$ (if the number of consecutive times of not waking exceeds $L_u=5$, the agent is forced to wake up). If the agent $i$ wakes at the instant $k$, it observes a realization of the random variable $X_k^i$, which we assume to be a truncated normal distribution with means $i$ and variance $1$. There are $m = 3$ hypothesis for each agent whose correspondent parameters and confidence values are presented in Fig.~\ref{fig:options}. In this example, the initial conditions are set as a uniform distribution for all the hypotheses, and the optimal solution is $\theta_3$ with $F(\theta^*)=0.29$. It is worth noting that agents $1$ and $4$ have the highest confidence values for $\theta_1$ individually; agent $2$ has the highest individual confidence in $\theta_2$. However, the network's optimal hypothesis is $\theta_3$; thus, cooperation is needed.

\begin{figure}[t!]
\centering
\includegraphics[width=0.4\textwidth]{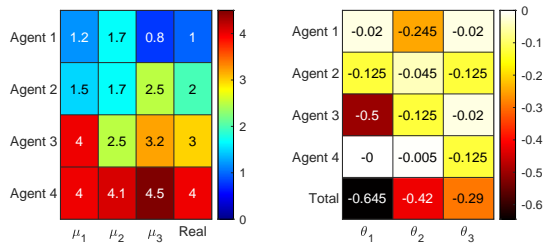}
\caption{\textbf{Left:} Mean of each likelihood function for each hypothesis, as well as the mean of the true distribution of the observations. \textbf{Right:} Function value of $F(\theta)$ for each hypotheses, as well as the value of the local function for each agent.}
\label{fig:options}
\end{figure}

We first show our results for a star topology with a waking probability $p_w=0.9$ and a link failure probability $p_l=0.2$ at each time $k$. Fig.~\ref{fig:star1} shows the evolution of the beliefs of each agent. As suggested by Theorem~\ref{thm:main}, all agents concentrate their beliefs on $\theta^*$ after a transitory period. It can be seen that the beliefs of the agents oscillate between different hypotheses considering their conflict between the individual best response and the network dynamics. However, after some time, oscillations attenuate and eventually disappear. If an agent is not connected for a period of time, its belief fluctuates, but once it connects with the network again, its beliefs concentrate again around the optimal hypotheses.

\begin{figure}[t!]
\centering
\small{\textbf{Star Graph, 4 Agents, $p_w=0.9$, $p_l=0.2$}}
\includegraphics[trim=5 24 10 24,clip,width=0.81\linewidth]{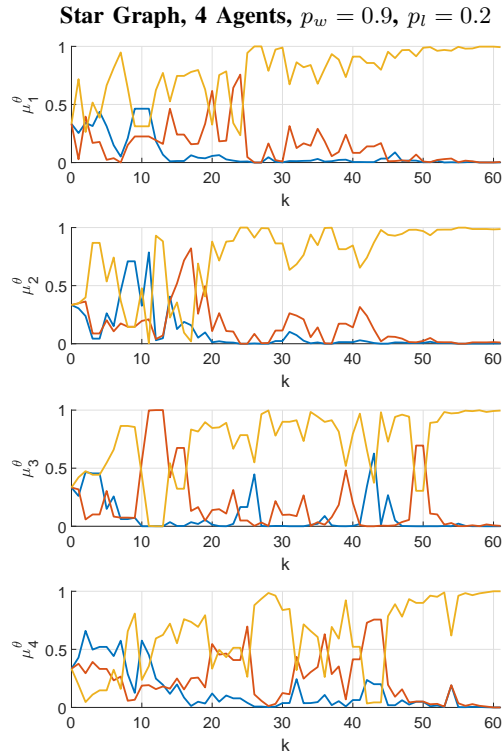}
\vspace{-0.5cm}
\caption{Beliefs values evolution of each agent regarding to each one of the 3 hypothesis (blue, red and yellow, respectively) with a star base graph, 90\% wake probability and 20\% link failure probability.}
\label{fig:star1}
\end{figure}

As the second network configuration, we continued with the base star topology and set $p_w=0.5$ and $p_l=0.1$ to illustrate the algorithm's behavior with higher delays between the send and receipt of the messages. Fig. \ref{fig:star2} presents the evolution of the beliefs with these new conditions. Once more, the algorithm converges to the optimal solution of the problem $\theta^*$, but in this case, after a longer transitory period, given the larger periods of sleep for the agents. It is also noticeable that the fluctuations produced after the convergence become frequent. However, even the strongest fluctuations observed eventually return to the optimal solution.

\begin{figure}[t!]
\centering
\small{\textbf{Star Graph, 4 Agents, $p_w=0.5$, $p_l=0.1$}}
\includegraphics[trim=5 24 10 24,clip,width=0.81\linewidth]{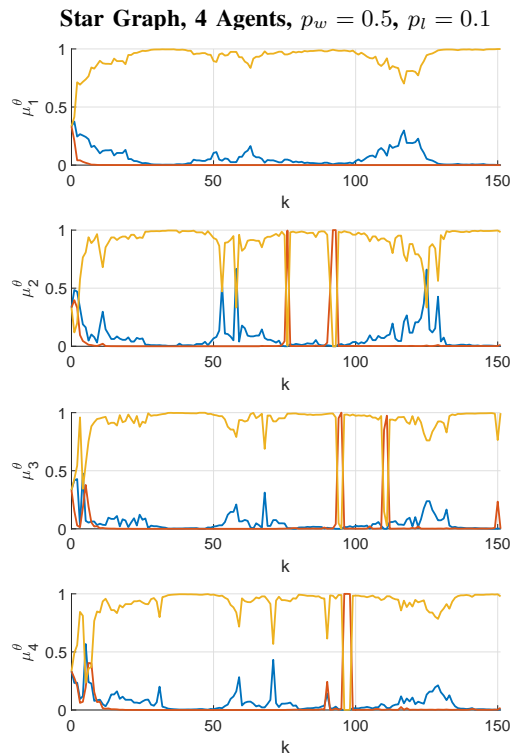}
\vspace{-0.5cm}
\caption{Beliefs values evolution of each agent regarding to each one of the 3 hypothesis (blue, red and yellow respectively) with a star base graph, 50\% wake probability and 10\% link failure probability.}
\label{fig:star2}
\end{figure}

Next, Fig.~\ref{fig:Path2} shows the behavior of Algorithm~\ref{alg:main} on a path graph with $p_w=0.5$ and $p_l=0.1$. In this scenario, we identify a behavior similar to the star configuration presented with the same parameters supporting the algorithm's network independence. Fig.~\ref{fig:Path1} shows the evolution of the beliefs with this topology and the parameters set to $p_w=0.9$ and $p_l=0.2$. A behavior similar to the base star configuration was observed. Concentration around the optimal hypothesis is achieved for all agents. Some minor fluctuations appear sporadically without affecting the final result.

\begin{figure}[t]
\centering
\small{\textbf{Path Graph, 4 Agents, $p_w=0.5$, $p_l=0.1$}}
\includegraphics[trim=5 24 10 24,clip,width=0.81\linewidth]{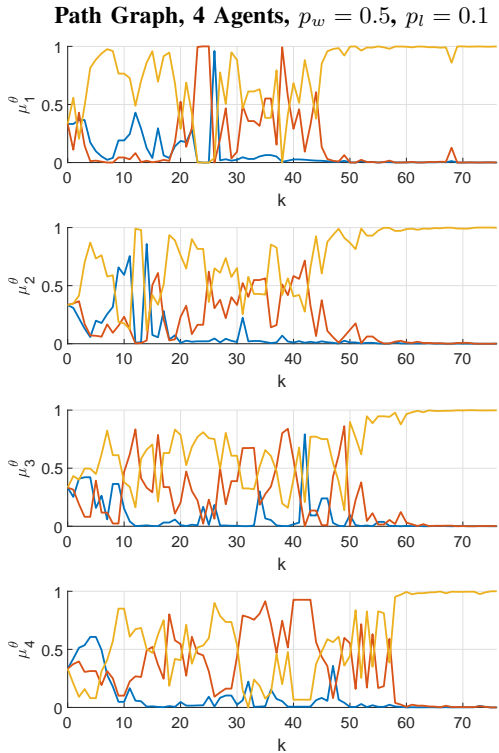}
\vspace{-0.5cm}
\caption{Beliefs values evolution of each agent regarding to each one of the 3 hypothesis (blue, red and yellow respectively) with a path base graph, 90\% wake probability and 20\% link failure probability.}
\label{fig:Path1}
\end{figure}

\begin{figure}[t]
\centering
\small{\textbf{Path Graph, 4 Agents, $p_w=0.9$, $p_l=0.2$}}
\includegraphics[trim=5 24 10 24,clip,width=0.81\linewidth]{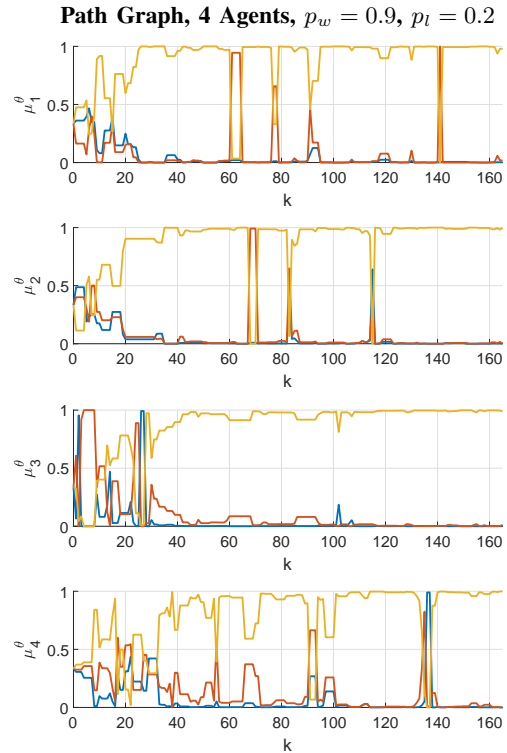}
\vspace{-0.5cm}
\caption{Beliefs values evolution of each agent regarding to each one of the 3 hypothesis (blue, red and yellow respectively) with a path base graph, 50\% wake probability and 10\% link failure probability.}
\label{fig:Path2}
\end{figure}

Finally, we consider a a cycle graph in Fig.~\ref{fig:Circle1} with $p_w=0.9$ and $p_l=0.2$. Figure \ref{fig:Circle2} shows the evolution of the beliefs in the circular topology with parameters $p_w=0.5$ and $p_l=0.1$. 

\begin{figure}[t]
\centering
\small{\textbf{Cycle Graph, 4 Agents, $p_w=0.9$, $p_l=0.2$}}
\includegraphics[trim=5 24 10 24,clip,width=0.81\linewidth]{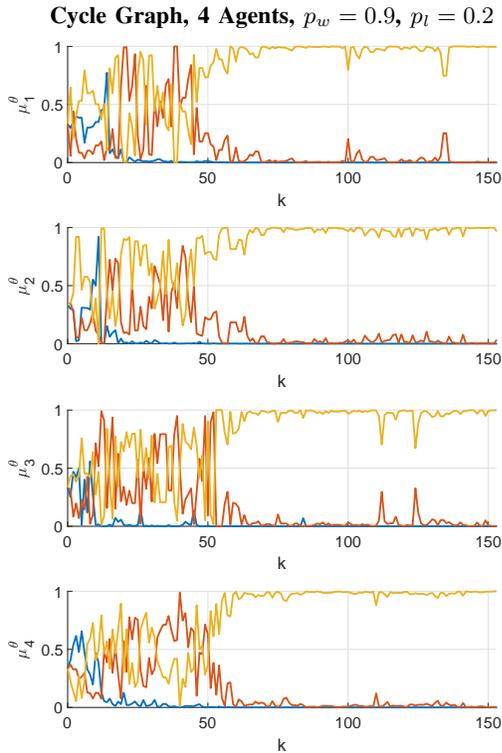}
\vspace{-0.5cm}
\caption{Beliefs values evolution of each agent regarding to each one of the 3 hypothesis (blue, red and yellow respectively) with a circular base graph, 90\% wake probability and 20\% link failure probability.}
\label{fig:Circle1}
\end{figure}

\begin{figure}[t]
\centering
\small{\textbf{Cycle Graph, 4 Agents, $p_w=0.5$, $p_l=0.1$}}
\includegraphics[trim=5 24 10 24,clip,width=0.81\linewidth]{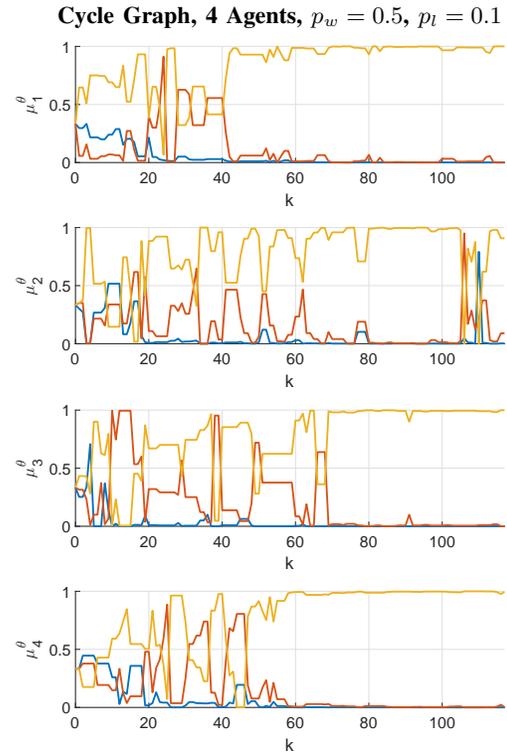}
\vspace{-0.5cm}
\caption{Beliefs values evolution of each agent regarding to each one of the 3 hypothesis (blue, red and yellow respectively) with a circular base graph, 50\% wake probability and 10\% link failure probability.}
\label{fig:Circle2}
\end{figure}
%
\section{Conclusions}\label{sec:conclusions}

We proposed a robust cooperative learning algorithm for distributed non-Bayesian learning that guarantees the asymptotic exponential concentration of the beliefs on the set of optimal hypotheses. Contrary to available non-Bayesian learning algorithms, we build upon recently proposed network models with relatively weak connectivity assumptions. These new connectivity models allow for asynchronous communications, message delays, unpredictable message losses, and directed communications. Our result shows that all network agents, regardless of their observation model or their relative centrality in the network, concentrate their beliefs, asymptotically exponentially fast. The beliefs concentrate at a network-independent rate and lower bounded by the average optimality gap between the optimal hypotheses set and the second-best hypotheses. 

Future work will consist of a non-asymptotic analysis of the belief concentration phenomena. Our main result shows that the belief concentration rate is lower bounded by the optimality gap between the optimal hypotheses set and the second-best hypotheses. Such lower bound is effectively zero is one considers a continuum of hypotheses or parameter space, which will be the natural setup for standard estimation problems such as parametric inference in the exponential family of distributions.


\bibliographystyle{./bibliography/IEEEtran}
\bibliography{./bibliography/IEEEabrv,./bibliography/IEEEACC2020}

\end{document}